\documentclass{article}

\usepackage{a4wide}
\usepackage{cite}
\usepackage{url}            
\usepackage{amsfonts,amsthm}       
\usepackage{epsfig}
\usepackage{epstopdf}
\usepackage[ruled,vlined]{algorithm2e}
\usepackage{mathtools,stmaryrd}
\usepackage{xparse}
\usepackage{hyperref} 
\usepackage{newtxtext}
\usepackage[frenchmath]{newtxmath}

\DeclarePairedDelimiterX{\segN}[2]{\dlb}{\drb}{{#1}\,{,}\,{#2}}

\DeclarePairedDelimiterX{\Iintv}[1]{\llbracket}{\rrbracket}{\iintvargs{#1}}

\NewDocumentCommand{\iintvargs}{>{\SplitArgument{1}{,}}m}{\iintvargsaux#1}

\NewDocumentCommand{\iintvargsaux}{mm} {#1\mkern1.5mu..\mkern1.5mu#2}

\newtheorem{theorem}{Theorem}
\newtheorem{proposition}[theorem]{Proposition}
\newtheorem{definition}[theorem]{Definition} 
\newtheorem{corollary}[theorem]{Corollary} 
\newtheorem{lemma}[theorem]{Lemma} 
\newtheorem{remark}[theorem]{Remark}

\begin{document}

\title{Regional Gradient Observability for Fractional Differential Equations with Caputo Time-Fractional Derivatives}

\author{Khalid Zguaid$^{1}$\\ 
\url{https://orcid.org/0000-0003-3027-8049}\\
\texttt{zguaid.khalid@gmail.com} 
\and Fatima-Zahrae El Alaoui$^{1}$\\ 
\url{https://orcid.org/0000-0001-8912-4031}\\
\texttt{fzelalaoui2011@yahoo.fr}
\and Delfim F. M. Torres$^{2,}$\thanks{Corresponding author.
Telephone: +351 234 370 668; fax: +351 234 370 066; e-mail: delfim@ua.pt\newline
This is a preprint of a paper whose final and definite form is published 
in 'Int. J. Dyn. Control' (ISSN 2195-268X).\newline 
Submitted 11/July/2022; Revised 07/Nov/22; and Accepted 26/Dec/2022.}\\ 
\url{https://orcid.org/0000-0001-8641-2505}\\
\texttt{delfim@ua.pt}}

\date{$^{1}$TSI Team, Department of Mathematics, Faculty of Sciences,\\
Moulay Ismail University, 11201 Meknes, Morocco\\[0.3cm]
$^{2}$\mbox{Center for Research and Development in Mathematics and Applications (CIDMA),}\\
Department of Mathematics, University of Aveiro, 3810-193 Aveiro, Portugal}
  	
\maketitle

% -----------------------------------------

\begin{abstract}
We investigate the regional gradient observability
of fractional sub-diffusion equations involving the Caputo derivative. 
The problem consists of describing a method to find and recover 
the initial gradient vector in the desired region, which 
is contained in the spacial domain. After giving necessary 
notions and definitions, we prove some useful characterizations 
for exact and approximate regional gradient observability. 
An example of a fractional system that is not (globally) 
gradient observable but it is regionally 
gradient observable is given, showing
the importance of regional analysis.
Our characterization of the notion of regional gradient observability
is given for two types of strategic sensors. 
The recovery of the initial gradient is carried out
using an expansion of the Hilbert Uniqueness Method. Two illustrative
examples are given to show the application of the developed approach. 
The numerical simulations confirm that the proposed algorithm is
effective in terms of the reconstruction error.   

\medskip

\end{abstract}

% -----------------------------------------

\section{Introduction}

The investigation of distributed parameter systems (DPS) drives 
many useful concepts in science and engineering, including the 
well-known notions of stability, controllability and observability
\cite{curtain.1978,zwart.1995,eljai.1997,eljai.2012,loins.1997,weiss.2009}. 
These concepts allow one to have a better understanding of the investigated system, 
enhancing the ability to control it. All these notions are important and have their 
particularities, but here we only focus on the concept of observability, 
which was firstly introduced by Kalman, for finite dimensional systems, 
in 1960 \cite{kalman.1960}. The goal is to recover an initial unknown state
using the output parameters or the measurements
of the considered system. After the pioneer work of Kalman, the concept 
of observability was also developed to cover infinite dimensional systems
\cite{magnusson.1984,wang.1982}. 

In the nineties of the 20th century, El Jai and others introduced
a more general notion called regional observability
\cite{amouroux.1994,simon.1993,zerrik.1993}.
Its main objective is to find and recuperate the unknown initial vector 
of the studied distributed parameter system, but only in a partial region 
of the spacial domain. The key advantage of regional observability becomes 
clear when the considered system is not (globally) observable in the 
whole spatial domain. In such cases, the studied system can be regionally 
observable in some well chosen sub-region. Thus, we can at least partially recover 
the initial state, which might be useful in many areas of science \cite{MR4376818}.

After the regional observability concept has been introduced,
Zerrik, Badraoui and El Jai proposed the notion of 
regional boundary observability, which has the same goal 
of regional observability but where the desired sub-region 
is a part of the boundary \cite{zerrik.bnd.200,zerrik.bnd.1999}. 
Although important, all such notions and results were not enough 
to get all possible characterizations of DPS. 
For this reason, in the 21th century the notion of 
regional gradient observability has been introduced
and investigated, with the goal of finding and recover 
the initial gradient vector in a suitable region
\cite{zerrik.grd.2003,zerrik.grd.2000}.
We adopt here this notion of gradient observability, 
which has been subject of a recent increase of interest 
\cite{grad.frac,MR3802082,MR3757160}. This is due to the fact 
that the concept of gradient observability finds applications 
in real-life situations. For instance, consider the problem 
of determining the laminar flux on the boundary of a heated vertical 
plate developed in steady state: see Figure~\ref{plate} 
for the profile of an active plate. In this case, the notion 
of gradient observability is associated to the problem 
of determining the thermal transfer that is generated by the heated plate.
For more on the subject, and for applications of the different observability 
concepts for various kinds of systems, we refer the reader to 
\cite{boutoulout.2010,boutoulout.2013,boutoulout.2015,boutoulout.2014}.
% --------------------------------------
\begin{figure}
\center{\includegraphics[width=\textwidth]{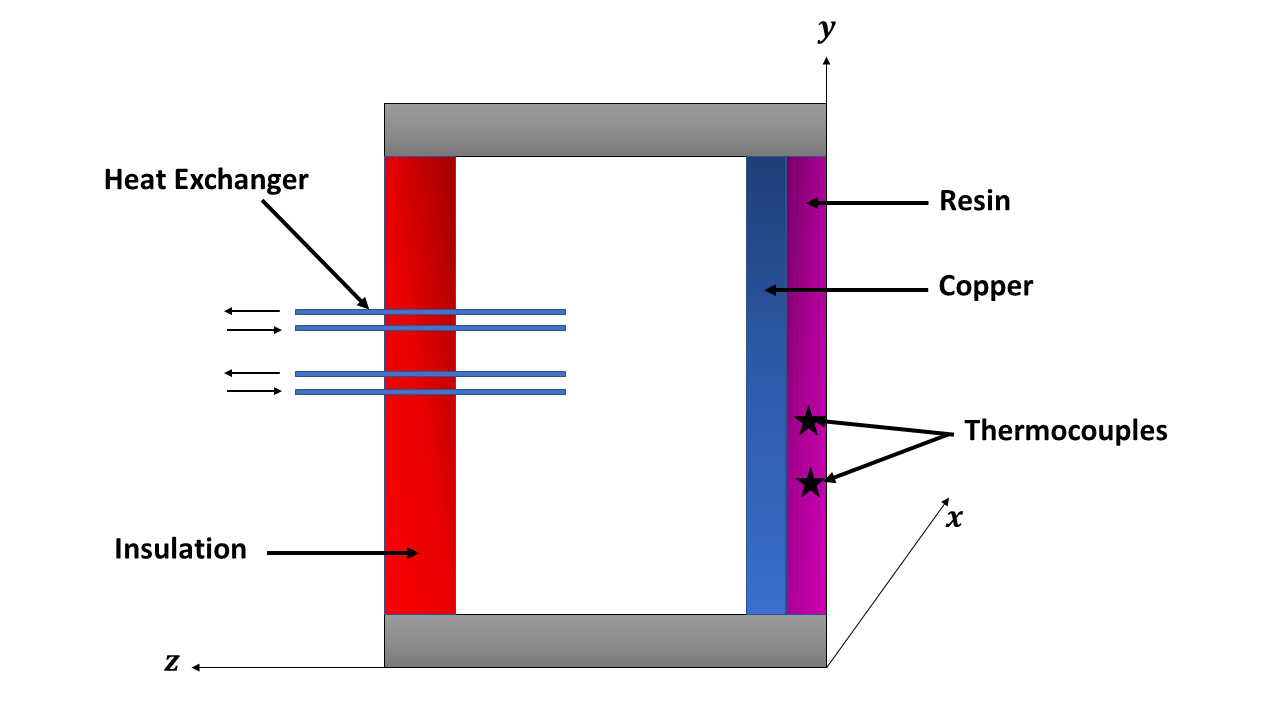}}
\caption{\label{plate} Profile of an active plate.}
\end{figure} 
% --------------------------------------

Fractional calculus is one of the most rapidly spreading 
domains in mathematics nowadays, especially the use of fractional-order systems 
to model real-world phenomena \cite{hand.bal.A,hand.bal.B,hand.petras,torres_roda.2021,torres_rosa.2022}.
It is well known that fractional operators, non-integer order differentiation and 
non-integer order integration operators, have many outstanding properties 
that make them fruitful and suitable for describing and studying 
the characteristics of certain real-world problems. The non-local fractional 
operators, not only consider the local points to calculate the (fractional) derivative 
of some function but also consider the past states, as is the case with left-sided fractional 
operators, or the future states, as happens with the right-sided operators. We also mention that
fractional operators have hereditary properties \cite{hand.taras.A,hand.taras.B}. Moreover, 
the diversity of fractional operators can also be seen as an advantage of fractional calculus 
because having many different types of fractional integrals and derivatives lead to more choices 
in the modeling of real-world phenomena. This explains why fractional calculus has been used
with success and benefit in various different domains. For more details, we refer 
the interested reader to the books \cite{torres_nda.2021,podlubny}.

In the subject of regional observability, we can already find several works dealing with fractional
systems \cite{MR3955081,me.atanaa,MyID:423,MR4066412,me.matcom.2022,zguaid.2022.2,me.sm2a}. However,
investigations of regional gradient observability for time-fractional diffusion processes are scarce. 
We are only aware of \cite{grad.frac}, where Ge, Chen and Kou propose a reconstruction procedure for
Riemann--Liouville fractional time diffusion processes. The main goal of the present paper is the
investigation of regional gradient observability for time-fractional diffusion systems described 
by the Caputo derivative, where the purpose is to find and reconstruct the gradient 
of the initial state of the considered system in a desired subregion of the evolution domain. 
This is in contrast with \cite{grad.frac}, where non-integer order systems are written with the
Riemann--Liouville derivative and where it is mentioned that their approach fails to cover systems
described by the Caputo derivative (cf. Lemma~7 of \cite{grad.frac}). Here we prove 
an alternative lemma that fixes the drawbacks mentioned in \cite{grad.frac}. 
Our contribution consists of giving several characterizations 
for regional exact and approximate gradient observability of the considered linear system. 
We present a method that allows the regional reconstruction of the initial gradient vector 
in the desired subregion. Moreover, we provide some simple numerical simulations that 
back-up our theoretical results.

The organization of our paper is done in the upcoming manner. 
In Section~\ref{sec:2}, the necessary background information 
about regional gradient observability, as well as some of its 
useful properties and characterizations, are given. Section~\ref{sec:3} 
is devoted to illustrate, throughout a counterexample, that 
we can have a system that is not gradient observable but it is 
regionally gradient observable in some suitable region included in the evolution space. 
A full characterization of the notion in hand, via gradient strategic sensors, 
is then given in Section~\ref{sec:4}, while in Section~\ref{sec:5} we develop 
the steps to be followed in order to achieve the regional flux reconstruction. 
In Sections~\ref{sec:6} and \ref{sec:7} we present, respectively, two applications 
of the obtained results and two successful numerical simulations.
Finally, we end with Section~\ref{sec:8} of conclusions and some future
directions of research.

% ----------------------------------------- 

\section{Problem Statement and Regional Gradient Observability}
\label{sec:2}

We now present a general formulation of the considered problem of initial 
gradient reconstruction. We also layout all the needed preliminary results 
and ingredients to make it easy for the reader to follow smoothly throw the manuscript. 

Let $\Omega$ be a connected, open, and bounded set in $\mathbb{R}^n$, $n\geq 1$, 
possessing a Lipschitz-continuous boundary $\partial\Omega$. For any final time 
$T\in\mathbb{R}^*_+$, we designate $Q_T := \Omega\times[0,T]$ and 
$\Sigma_T := \partial\Omega\times[0,T]$. Let us take the dynamic of the 
considered system to be the following operator 
defined in the state space $E=L^{^2}(\Omega)$ as
\begin{equation}
\label{A}
\mathcal{D}(A) = H^2(\Omega)\cap H^1_0(\Omega) 
\ \mbox{ and }  \ 
Ay(x) = \displaystyle\sum_{k,l=1}^{n} 
\partial_{x_k}\left(a_{k,l}(x)\partial_{x_l}y(x)\right), 
\quad \forall x\in\Omega, \ \forall y\in E,
\end{equation}
where $\partial_{x}$ stands for $\dfrac{\partial}{\partial x}$
and the coefficients $a_{i,j} \in C^1(\overline{\Omega})$ 
satisfy the following hypotheses:
\begin{enumerate}

\item[$(\mathcal{H}_1)$] $a_{k,l}(\cdot) = a_{l,k}(\cdot)$;

\item[$(\mathcal{H}_2)$] $\exists \mu\in\mathbb{R}$ such that: 
$$
\displaystyle\sum_{k,l=1}^{n} a_{k,l}(x)\varsigma_k\varsigma_l
\geq \mu \|\varsigma\|^2, \ x\in\Omega,
$$ 
for $\varsigma=(\varsigma_1,\ldots,\varsigma_n)\in\mathbb{R}^n$ 
and where $\|\varsigma\|=\sqrt{\varsigma_1^2 + \cdots + \varsigma_n^2}$. 
\end{enumerate}

Hypotheses $(\mathcal{H}_1)$ and $(\mathcal{H}_2)$ mean, respectively, 
that $A$ is symmetric and $-A$ is uniformly elliptic. In this case, 
it is well known that $-A$ has a set of eigenvalues such that (see \cite{Yam}):
$\left(\lambda_i\right)_{i\geq 1}$, 
$$
0<\lambda_1<\lambda_2
< \cdots < \lambda_i
<\cdots \rightarrow +\infty
$$  
Each eigenvalue $\lambda_i$ 
corresponds with $r_i$ eigenfunctions 
$\left\{\varphi_{i,j}\right\}_{1\leq j \leq r_i}$, 
where $r_i\in\mathbb{N}^*$ is the multiplicity 
of $\lambda_i$, such that $A\varphi_{i,j} = \lambda_i\varphi_{i,j}$ 
and $\varphi_{i,j} \in H^2(\Omega)\cap H^1_0(\Omega)$, 
$\forall i\in\mathbb{N}^*$ and $1\leq j\leq r_i$. Furthermore, the set 
$\left\{\varphi_{i,j}\right\}_{\substack{ i\geq 1 \\ 1\leq j\leq r_i}}$ 
constitute an orthonormal basis of $E$.

The operator $A$ is an infinitesimal generator of a $C_0$-semigroup 
$\left\{\mathcal{S}(t)\right\}_{t\geq 0}$ on $E$, written as:
\begin{equation}
\label{S.G}
\mathcal{S}(t)y(x) = \displaystyle\sum_{i=1}^{+\infty}\exp(-\lambda_it)
\sum_{j=1}^{r_i}\langle y,\varphi_{i,j}\rangle\varphi_{i,j}(x), 
\ \forall y\in E.
\end{equation}
Here we study fractional systems possessing the form:
\begin{equation}
\label{sys}
\left\{\begin{array}{llll}
^{^C}\mathcal{D}_{0^+}^{^\alpha}u(x,t) =  Au(x,t),  
& (x,t)\in Q_T, \\ 
u(\xi,t) = 0,  & (\xi,t)\in \Sigma_T, \\
u(x,0) = u_0(x)\in H^1_0(\Omega),  
& x\in\Omega,
\end{array}\right.
\end{equation}
where 
$^{^C}\mathcal{D}_{0^+}^{^\alpha}u(\cdot,t) 
:= \displaystyle\int_{0}^{t}\dfrac{(t-e)^{-\alpha}}{\Gamma(1-\alpha)}
\partial_e u(\cdot,e)de$ is the fractional derivative of $u$, in Caputo's 
sense, and $\Gamma(\alpha) = \displaystyle\int_{0}^{+\infty}t^{\alpha-1}e^{-t}dt$ 
is the Euler gamma function. For $u_0\in H^1_0(\Omega)$, both its value and its 
gradient are supposedly unknown. System \eqref{sys} has one and only one 
mild solution in $C(0,T;E)\cap L^2(0,T;\mathcal{D}(A))$ of the following form: 
\begin{equation}
\label{sol}
u(\cdot,t) = \mathcal{S}_\alpha(t)u_0(\cdot) 
:= \displaystyle\sum_{i=1}^{\infty}
E_\alpha(-\lambda_it^\alpha)\sum_{j=1}^{r_i}\langle u_0,
\varphi_{i,j}\rangle\varphi_{i,j}(\cdot), \ 0\leq t\leq T,
\end{equation}
where $E_\alpha(z) = \displaystyle
\sum_{k=0}^{\infty}\dfrac{z^k}{\Gamma(\alpha k+1)}$ 
stands for the one parameter Mittag--Leffler function \cite{Yam1}.

Without any loss of generality, we take $u(t) := u(\cdot,t)$.
The output function, which provides measurements 
and information on the consider system, is:
\begin{equation}
\label{out}
z(t) = Cu(t), \quad 0\leq t\leq T.
\end{equation}
The operator $C: E \rightarrow \mathcal{O}$ satisfies the following 
admissibility condition for $\mathcal{S}_\alpha$ \cite{zguaid.2021}:
\begin{equation}
\label{adm.cnd}
\exists M>0, \quad \displaystyle\int_{0}^{T}\|C
\mathcal{S}_\alpha(t)v\|_{\mathcal{O}}^{^2}dt 
\leq M\|v\|_{E}^{^2}, \  \forall v\in\mathcal{D}(A).
\end{equation}
The Hilbert space $\mathcal{O}$ is called the observation space.

The operator $\mathcal{S}_\alpha(t)$ defined in \eqref{sol}
is a linear bounded operator, see \cite{zguaid.2021}, which describes the evolution 
of the considered time-fractional system in function of its initial state. 
Moreover, if the operator $C$ is bounded, then the admissibility condition 
\eqref{adm.cnd} is always satisfied, which means that any bounded observation 
operator is an admissible observation operator.

Let $\omega\subset\Omega$ be the desired sub-region. 
We introduce the following restriction operators: 
$$
\begin{array}{lllll}
\chi_{_\omega} &:& E&\longrightarrow & L^2(\omega)\\
& & \hfill u & \longmapsto & u_{|_\omega}
\end{array} 
\ \mbox{ and } \ 
\begin{array}{lllll}
\chi_{_\omega}^{n} &:& E^{^n}&\longrightarrow & (L^2(\omega))^{^n}\\
& & \hfill  u& \longmapsto & u_{|_\omega} = \left(\chi_{_\omega}u_1,
\chi_{_\omega}u_2,\ldots,\chi_{_\omega}u_n\right),
\end{array} 
$$
and their adjoint,
$$
\begin{array}{lllll}
\chi_{_\omega}^* &:& L^2(\omega)&\longrightarrow & E\\
& & \hfill v & \longmapsto & 
\left\{
\begin{array}{ll}
v & \text{in} \ \omega\\
0& \text{in} \ \Omega\setminus\omega
\end{array}\right.
\end{array} 
\ \mbox{ and } \ 
\begin{array}{lllll}
(\chi_{_\omega}^{n})^* &:& (L^2(\omega))^{^n}&\longrightarrow & E^{^n}\\
& & \hfill  v& \longmapsto &  \left(\chi_{_\omega}^*v_1,
\chi_{_\omega}^*v_2,\ldots,\chi_{_\omega}^*v_n\right).
\end{array} 
$$
Substituting \eqref{sol} in \eqref{out} gives, 
$$
z(t) = C\mathcal{S}_\alpha(t)u_0
:=(\mathcal{K}_\alpha u_0)(t), \quad t\in[0,T],
$$
where $\mathcal{K}_\alpha : \mathcal{D}\left(\mathcal{K}_\alpha\right)
\subset E \longrightarrow L^2(0,T;\mathcal{O})$ is the observability 
operator, which has an important contribution in defining and characterizing 
both regional and regional gradient observability. 
In \cite{regional.analysis}, the admissibility hypothesis on $C$ 
makes it possible to express the adjoint of $\mathcal{K}_\alpha$ as:
\begin{equation}
\label{K_alp}
\begin{array}{lllll}
\mathcal{K}_\alpha^*&:& \mathcal{D}\left(\mathcal{K}_\alpha^*\right)
\subset L^2(0,T;\mathcal{O}) &\longrightarrow &E,\\
& &\hfill z& \longmapsto
& \displaystyle\int_{0}^{T}\mathcal{S}_\alpha^*(e)C^*z(e)de. 
\end{array}
\end{equation}
Let $\nabla : \mathcal{D}(\nabla)\subset E \rightarrow E^n$ 
be the gradient operator, $\nabla v = \left(\partial_{x_1} v,
\partial_{x_2} v,\ldots,\partial_{x_n} v\right)$, 
for all $ v$ in $ \mathcal{D}(\nabla) = H_0^1(\Omega)$. 
As shown in \cite{Kurula.2012}, the adjoint of $\nabla$ 
is minus the divergence, that is, $\forall V\in\mathcal{D}(\nabla^*)$,
$$
\nabla^*V 
= \left\{
\begin{array}{ll}
-\operatorname{div}(V) 
:= -\displaystyle\sum_{i=1}^{n}\partial_{x_i} V_i 
& \text{in} \quad \Omega,\\
\qquad 0& \text{on} \quad \partial\Omega.
\end{array}\right.
$$
The initial state can be decomposed as follows: 
$$
u_0=\left\{
\begin{array}{ll}
u_0^1& \text{in} \ \omega,\\
\tilde{u}_0 & \text{in} \ \Omega\setminus\omega.
\end{array}\right.
$$
The purpose of regional gradient observability is to reconstruct 
the gradient vector $\nabla u_0^1$ in $\omega$.

We recall that system \eqref{sys} augmented with \eqref{out} 
is called exactly (respectively, approximately) $\omega$-observable 
if, and only if, $Im\left(\chi_{_\omega}\mathcal{K}_\alpha^*\right) 
= L^2(\omega)$ (respectively, $Ker\left(\mathcal{K}_\alpha\chi_{_\omega}^*\right) 
= \left\{0\right\}$). Based on the discussion 
in \cite{zerrik.grd.2003}, we denote 
$H_\alpha := \chi_{_\omega}^n\nabla\mathcal{K}_\alpha^*$ 
and we enunciate the following definitions.

\begin{definition}
\label{def:d1}
System \eqref{sys}, augmented with \eqref{out}, 
is exactly $G$-observable in $\omega$ 
($G$ stands for Gradient) if 
\begin{equation}
\label{exact.G}
Im\left(H_\alpha\right) = \left(L^2(\omega)\right)^n.
\end{equation}
\end{definition}

\begin{definition}
\label{def:d2}
System \eqref{sys}, augmented with \eqref{out}, 
is approximately $G$-observable in $\omega$ if 
\begin{equation}
\label{app.G}
\overline{Im\left(H_\alpha\right)} 
= \left(L^2(\omega)\right)^n.
\end{equation}
\end{definition}

\begin{remark} 
Definitions~\ref{def:d1} and \ref{def:d2} for fractional systems 
coincide with the standard notions for the classical systems 
in the particular case $\alpha=1$ \cite{zerrik.grd.2003}.
\end{remark}

We now present some useful results and properties. 
Our first result gives a characterization 
of the approximate regional gradient observability.

\begin{proposition}
\label{prp1}
The upcoming assertions are equivalent:
\begin{enumerate}
\item[1-] System \eqref{sys} is approximately $G$-observable in $\omega$.
\item[2-] $Ker\left(H_\alpha^*\right)$ = $\left\{0\right\}$.
\item[3-] $H_\alpha H_\alpha^*$ is positive definite.
\item[4-] $\left(\langle\left(\chi_{_\omega}^n\right)^*y, 
\nabla\mathcal{K}_\alpha^*z \rangle_{_{\left(E\right)^n}} = 0, 
\ \forall z\in L^2(0,T;\mathcal{O})\right)$ 
$\implies$ $y=0_{_{\left(L^2(\omega)\right)^n}}$. 
\end{enumerate}
\end{proposition} 

\begin{proof}
The result follows by proving that $1) \iff 2)$, 
$2) \implies 3)$, $3) \implies 4)$ and $4)\implies 2)$.
\begin{description}
\item[$1) \iff 2)$] This is a direct consequence from the fact that
$Ker(H_\alpha^*) = (Im(H_\alpha))^{^\perp}$.

\item[$2) \implies 3)$] Let $y$ be in $\left(L^2(\omega)\right)^n$. Then,
$$
\langle H_\alpha H_\alpha^*y,y\rangle_{_{\left(L^2(\omega)\right)^n}} 
= \langle H_\alpha^*y,H_\alpha^*y\rangle_{_{L^2(0,T;\mathcal{O})}} 
= \|H_\alpha^*y\|_{_{L^2(0,T;\mathcal{O})}}^2 \geq 0.
$$
Moreover, we get that,
$$
\langle H_\alpha H_\alpha^*y,y\rangle_{_{\left(L^2(\omega)\right)^n}}=0 
\ \implies \ H^*_\alpha y=0,
$$
and, using $2)$, we have:
$$
\langle H_\alpha H_\alpha^*y,y\rangle_{_{\left(L^2(\omega)\right)^n}}
=0 \ \implies \  y=0.
$$
Thus, $H_\alpha H_\alpha^*$ is positive definite. 

\item[$3) \implies 4)$] Let us consider $y\in \left(L^2(\omega)\right)^n$ 
such that $\langle\left(\chi_{_\omega}^n\right)^*y, 
\nabla\mathcal{K}_\alpha^*z \rangle_{_{\left(E\right)^n}} 
= 0$, for all $z\in L^2(0,T;\mathcal{O})$.
Thus, by choosing $z=H_\alpha^*y$, we obtain that
$$
\langle\left(\chi_{_\omega}^n\right)^*y, 
\nabla\mathcal{K}_\alpha^*H_\alpha^*y 
\rangle_{_{\left(E\right)^n}} = \langle H_\alpha^*y, 
H_\alpha^*y \rangle_{_{L^2(0,T;\mathcal{O})}} 
= \langle H_\alpha H_\alpha^*y, 
y \rangle_{_{\left(L^2(\omega)\right)^n}}=0.
$$
Hence, $3)$ implies that $y=0_{_{\left(L^2(\omega)\right)^n}}$. 

\item[$4)\implies 2)$] Let $y\in Ker\left(H_\alpha^*\right)$. 
We have $H_\alpha^*y=0$, which means that 
$\langle H_\alpha^* y,z\rangle_{_{L^2(0,T;\mathcal{O})}} =0$
for all $z\in L^2(0,T;\mathcal{O})$. Hence, 
$\langle \left(\chi_{_\omega}^n\right)^*y,
\nabla\mathcal{K}_\alpha^*z\rangle_{_{\left(E\right)^n}} =0$ 
for all $z\in L^2(0,T;\mathcal{O})$.
Thus, $4)$ implies that $y=0_{_{\left(L^2(\omega)\right)^n}}$ 
and we conclude that $Ker\left(H_\alpha^*\right) = \left\{0\right\}$.
\end{description}
The proof is complete.
\end{proof}

Before proving our second result, we recall the following lemma.

\begin{lemma}[See \cite{curtain.1978}]
\label{lemma}
Let $F$, $G$ and $E$ be three reflexive Banach spaces. 
Let us consider $v\in\mathcal{L}(F,E)$ and $y\in \mathcal{L}(G,E)$. 
The upcoming assertions are equivalent:
\begin{enumerate}
\item $Im(v) \subset Im(y)$;
\item $\exists$ $c>0$ such that:
$$
\|v^*x^*\|_{F^{^*}} \leq c \|y^*x^*\|_{G^{^*}}, \ \forall x^*\in E^{^*}.
$$
\end{enumerate}
\end{lemma}

The next proposition characterizes the exact regional gradient observability.

\begin{proposition}
\label{prp.2}
The mentioned statements are equivalent:
\begin{enumerate}
\item[1-] System \eqref{sys} is exactly $G$-observable in $\omega$;

\item[2-] $Ker\left(H_\alpha^*\right)$ 
= $\left\{0\right\}$ and $Im\left(H_\alpha\right)$ is closed;

\item[3-] There exists $c>0$ satisfying: 
$$
\|u\|_{_{\left(L^2(\omega)\right)^n}} 
\leq c\|H_\alpha^* u\|_{_{L^2(0,T;\mathcal{O})}}, 
\quad \forall u\in \left(L^2(\omega)\right)^n.
$$ 
\end{enumerate}
\end{proposition}

\begin{proof} 
We show that $1)\implies 2)$, $2) \implies 1)$, and $1) \iff 3)$.
\begin{description}
\item[$1)\Rightarrow 2)$] Since system \eqref{sys} is exactly 
$G$-observable in $\omega$, then it is also approximately 
$G$-observable in $\omega$. Hence, 
$Ker\left(H_\alpha^*\right) = \left\{0\right\}$ 
and $Im(H_\alpha) = \left(L^2(\omega)\right)^n 
= \overline{Im(H_\alpha)}$. Thus,  
$Ker\left(H_\alpha^*\right) = \left\{0\right\}$ 
and $Im(H_\alpha) $ is closed.

\item[$2) \Rightarrow 1)$] The equality 
$Ker\left(H_\alpha^*\right) = \left\{0\right\}$ gives that 
\eqref{sys} is approximately $G$-observable in $\omega$. This, 
together with the fact that $Im(H_\alpha)$ is closed, imply 
that $Im(H_\alpha) = \overline{Im(H_\alpha)} 
= \left(L^2(\omega)\right)^n$. Thus, system \eqref{sys} 
is exactly $G$-observable in $\omega$.
 
\item[$1) \Leftrightarrow 3)$] System \eqref{sys} is exactly 
$G$-observable in $\omega \Leftrightarrow Im\left(H_\alpha\right) 
= \left(L^2(\omega)\right)^n$.
We already know that $Im\left(H_\alpha\right) \subset \left(L^2(\omega)\right)^n$, 
hence all that remains is to show that 
$\left(L^2(\omega)\right)^n \subset Im\left(H_\alpha\right)$. 
This last inclusion is a direct application of Lemma~\ref{lemma} 
with $E=F=\left(L^2(\omega)\right)^n$, 
$G=L^2(0,T;\mathcal{O})$, 
$v=Id_{_{\left(L^2(\omega)\right)^n}}$, 
and $y=H_\alpha.$
\end{description}
The proof is complete.
 \end{proof}

\begin{remark}
Our Propositions~\ref{prp1} and \ref{prp.2}
generalize the main results of \cite{zerrik.grd.2003}, 
which are only valid for the classical integer-order case $\alpha=1$.
\end{remark}

% -----------------------------------------

\section{A Counterexample}
\label{sec:3}

To show the importance of regional gradient observability, 
we now give an example of a system that is not approximately 
gradient observable, but it is approximately 
$G$-observable in $\omega$.

Let us set $\Omega = ]0,1[\times]0,1[$ and let us work 
with the time-fractional system given by
\begin{equation}
\label{sys.exp}
\left\{\begin{array}{llll}
^{^C}\mathcal{D}_{0^+}^{^{0.5}}u(y_1,y_2,t) 
=  \partial_{y_1}^2u(y_1,y_2,t) + \partial_{y_2}^2u(y_1,y_2,t),  
& (y_1,y_2,t)\in Q_2, \\ 
u(\nu_1,\nu_2,t) = 0,  & (\nu_1,\nu_2,t)\in \Sigma_2, \\
u(y_1,y_2,0) = u_0(y_1,y_2),  & (y_1,y_2)\in\Omega,
\end{array}\right.
\end{equation}
together with the output
\begin{equation}
\label{out.exp}
z(t) = Cu(t) 
= \displaystyle\iint_{_D}u(y_1,y_2,t)f(y_1,y_2)dy_1dy_2,
\end{equation}
where $f(y_1,y_2) = \sin(2\pi y_2)$ and 
$D = \left\{\frac{1}{2}\right\}\times]0,1[$.

We know that the eigenvalues and eigenfunctions of 
$-A = -\partial_{y_1}^2 - \partial_{y_2}^2$ are written as follows:
$$
\lambda_{i,j} = (i^2 + j^2)\pi^2,
$$
and 
$$
\varphi_{i,j}(y_1,y_2) = 2\sin(i\pi y_1)\sin(j\pi y_2).
$$ 
Moreover, from \eqref{sol} and \eqref{out.exp}, one can write that:
\begin{equation}
\label{K_alp_count}
\mathcal{K}_{_\alpha}(t)u_0 = C \mathcal{S}_{_\alpha}(t)u_0 
= \displaystyle\sum_{i,j=1}^{+\infty}E_{0.5}(-\lambda_{i,j}t^{0.5})
\langle u_0,\varphi_{i,j}\rangle 
\displaystyle\iint_{_D}\varphi_{i,j}(y_1,y_2)f(y_1,y_2)dy_1dy_2.
\end{equation}
Let $h(y_1,y_2) = \dfrac{1}{4\pi}\left(\cos(y_1\pi)\sin(4y_2\pi), 
\frac{1}{4}\sin(y_1\pi)\cos(4y_2\pi)\right)$ be an element of $E^2$.

\begin{proposition}
The gradient $h$ is not approximately $G$-observable in $\Omega$ 
but it is approximately $G$-observable in 
$\omega = ]0,1[\times]\frac{1}{8},\frac{5}{8}[$.
\end{proposition}

\begin{proof}
Firstly, let us show that $h$ is not approximately 
$G$-observable in $\Omega$, i.e., 
$h\in Ker\left(\mathcal{K}_{_\alpha}(t)\nabla^{^*}\right)$
for all $t\in[0,T]$. We have:
$$
\begin{array}{lll}
\mathcal{K}_{_\alpha}(t)\nabla^{^*}h 
&=& \displaystyle\sum_{i,j=1}^{+\infty}
E_{0.5}(-\lambda_{i,j}t^{0.5})\langle \nabla^{^*}h,\varphi_{i,j}\rangle
\displaystyle\iint_{_D}\varphi_{i,j}(y_1,y_2)f(y_1,y_2)dy_1dy_2,\\
& = & \displaystyle\sum_{i,j=1}^{+\infty}
E_{0.5}(-\lambda_{i,j}t^{0.5})\int_{0}^{1}\sin(y_1\pi)
\sin(iy_1\pi)dy_1\int_{0}^{1}\sin(4y_2\pi)\sin(jy_2\pi)dy_2\\
& & \qquad \times \sin(\frac{i\pi}{2})
\displaystyle\int_{0}^{1}\sin(2y_2\pi)\sin(jy_2\pi)dy_2 = 0.
\end{array} 
$$
Hence, $h\in Ker\left(\mathcal{K}_{_\alpha}(t)\nabla^{^*}\right)$.

We now show that $h$ is approximately $G$-observable 
in $\omega$, i.e., $h \notin Ker\left(
\mathcal{K}_{_\alpha}(t)\nabla^{^*}(
\chi_{_\omega}^n)^{^*}\chi_{_\omega}^n\right)$ 
for all $t\in[0,T]$. We have:
\begin{equation*}
\begin{split}
\mathcal{K}_{_\alpha}(t)\nabla^{^*}(\chi_{_\omega}^n)^{^*}\chi_{_\omega}^nh 
&= \displaystyle\sum_{i,j=1}^{+\infty}E_{0.5}(-\lambda_{i,j}t^{0.5})\langle
\nabla^{^*}(\chi_{_\omega}^n)^{^*}\chi_{_\omega}^nh,\varphi_{i,j}\rangle
\displaystyle\iint_{_D}\varphi_{i,j}(y_1,y_2)f(y_1,y_2)dy_1dy_2,\\
&=\displaystyle\sum_{i,j=1}^{+\infty}E_{0.5}(-\lambda_{i,j}t^{0.5})
\int_{0}^{1}\sin(y_1\pi)\sin(iy_1\pi)dy_1\int_{\frac{1}{8}}^{\frac{5}{8}}
\sin(4y_2\pi)\sin(jy_2\pi)dy_2\\
& \qquad \times \sin(\frac{i\pi}{2})
\displaystyle\int_{0}^{1}\sin(2y_2\pi)\sin(jy_2\pi)dy_2,\\
&= E_{0.5}(-\lambda_{1,2}t^{0.5})\displaystyle\int_{0}^{1}\sin(y_1\pi)^2dy_1
\int_{\frac{1}{8}}^{\frac{5}{8}}\sin(4y_2\pi)\sin(2y_2\pi)dy_2
\int_{0}^{1}\sin(2y_2\pi)^2dy_2,\\
&= -\dfrac{\sqrt{2}}{24\pi}E_{0.5}(5\pi^2t^{0.5})\neq 0.
\end{split} 
\end{equation*}
We conclude that $h$ is approximately $G$-observable in $\omega$.
\end{proof}

% ----------------------------------------- 

\section{Gradient Strategic Sensors}
\label{sec:4}

In this section, we give a characterization of strategic gradient sensors 
whenever the considered system is approximately $G$-observable 
in the desired subregion.

\begin{definition}
We call a sensor any element $(D,f)$, where $D$ is the geometrical placement 
of the sensor, which is included in $\Omega$, 
and $f: D\rightarrow \mathbb{R}$ is its distribution. 
 \end{definition}

We introduce here two types of sensors:
\begin{itemize}
\item Zonal sensor, when $D$ has positive Lebesgue measure, $f\in L^2(D)$,
the space $\mathcal{O}$ is $\mathbb{R}$, and the measurements are given by 
$z(t)=\langle f,u(t)\rangle_{_{L^2(D)}} = \displaystyle\int_{D}f(x)u(x,t)dx$;

\item Pointwise sensor, when $D=\left\{b\right\}\in \Omega$, $f \equiv \delta_b$ 
with $\delta_b$ is the Dirac mass centered in $b$, 
the space $\mathcal{O}$ is $\mathbb{R}$, and the output equation 
is given by $z(t) =  u(b,t)$. 
\end{itemize}

\begin{remark}
When we consider a zonal sensor, then the observation operator is bounded; 
if we take a pointwise sensor, then the observation operator is unbounded 
but it is an admissible observation operator.
\end{remark}

For more information about sensors and their characterizations 
see \cite{Jai.Prit,regional.analysis,sensor.2002}. 

Let us reconsider system \eqref{sys}. We take the measurements to be given 
by means of $p$ sensors $\left(D_i,f_i\right)_{1\leq i\leq p}$. The observation 
space is $\mathcal{O}=\mathbb{R}^p$ and the output equation is written as:
\begin{equation}
\label{out.p}
z(t) = \left(z_1(t), \ldots , z_p(t)\right)^t,
\end{equation}
where $z_i(t) = \langle u(t),
f_i\rangle_{_{L^2(D)}}, \ \forall i\in \segN{1}{n}$.
The adjoint of the observation operator $C$ is expressed 
for all $y=(u_1,\ldots,u_p)\in\mathbb{R}^p$ by:
\begin{equation}
\label{C*z}
C^*u  = \displaystyle\sum_{i=1}^{p}\chi_{_{D_i}}f_iu_i, 
\end{equation}
for the case of zonal sensors, and by
\begin{equation}
\label{C*p}
C^*u  = \displaystyle\sum_{i=1}^{p}u_i\delta_{b_i},
\end{equation}
for the case of pointwise sensors.

\begin{definition}
A sequence of sensors (or a sensor) is said to be gradient $\omega$-strategic 
if \eqref{sys}, augmented with \eqref{out.p}, is approximately 
$G$-observable in $\omega$.
\end{definition}

In \cite{grad.frac}, it is given that a lemma (Lemma~7 of \cite{grad.frac})
that fails to be valid when the considered system is written in terms 
of the Caputo derivative, as we do here. Now we present an alternative 
new lemma that allows to deal with the problem.

\begin{remark}
The problem of this article is formulated with Caputo-type 
fractional derivatives only. However, Riemann--Liouville fractional derivatives 
appear naturally due to fractional integration by parts and Green's formulas 
(cf. Lemmas~\ref{lem:FIBP} and \ref{lem:FGF} below).
\end{remark}

\begin{lemma}
\label{lemma.new}
Let $r$ be a function that satisfies
\begin{equation}
\label{sys.RL}
\left\{\begin{array}{llll}
^{^{RL}}\mathcal{D}_{T^-}^{^\alpha}r(y,s) 
=  A^*r(y,s) + C^*z(s), & (y,s)\in Q_T, \ \alpha\in]0,1], \\ 
r(\xi,s) = 0,  & (\xi,s)\in \Sigma_T, \\
\lim\limits_{s\rightarrow T^-}
\mathcal{I}_{_{T^-}}^{^{1-\alpha}}r(y,s) = 0, 
& y\in\Omega,
\end{array}\right.
\end{equation}
where:
$$
^{^{RL}}\mathcal{D}_{T^-}^{^\alpha}r(y,s) 
= \partial_s\displaystyle\int_{s}^{T}
\dfrac{(e-s)^{-\alpha}}{\Gamma(1-\alpha)}r(y,e)de,
$$ 
is the right-sided fractional derivative in the sense of Riemann--Liouville, and,
$$
\mathcal{I}_{_{T^-}}^{^{\alpha}}r(y,s) 
= \dfrac{1}{\Gamma(\alpha)}\displaystyle\int_{s}^{T}(e-s)^{\alpha-1}r(y,e)de,
$$ 
is the right-sided Riemann-Liouville fractional integral. 
Then the following equality holds:
$$
\mathcal{K}_\alpha^*z = \mathcal{I}_{_{T^-}}^{^{1-\alpha}}r(x,0).
$$
\end{lemma}

\begin{proof}
The solution of \eqref{sys.RL} can be written as:
$$
r(s) = \displaystyle\int_{s}^{T}(e-s)^{\alpha-1}\mathcal{N}^*_\alpha(e-s)C^*z(e)de,
$$
where $\mathcal{N}_\alpha$ is a linear and bounded operator defined in terms 
of a probability density function \cite{zguaid.2021}.
From Proposition~3.3 of \cite{zguaid.2021}, we have that:
$$
\mathcal{I}_{_{T^-}}^{^{1-\alpha}}r(x,0) 
= \displaystyle\int_{0}^{T}\mathcal{S}_\alpha^*(\tau)C^*zd\tau.
$$
Hence, from \eqref{K_alp}, we have that:
$$
\mathcal{K}_\alpha^*z 
= \displaystyle\int_{0}^{T}\mathcal{S}_\alpha^*(\tau)C^*z(s)ds 
= \mathcal{I}_{_{T^-}}^{^{1-\alpha}}r(x,0),
$$
and the result is proved.
\end{proof}

The following fractional integration by parts formula 
will be useful in the sequel.

\begin{lemma}[See \cite{frac.intg}]
\label{lem:FIBP}	
Let $v$ be a function in $L^{^p}(0,T;E)$, let $u$ be in $AC(0,T;E)$ 
and $\alpha$ in $]0,1]$. The formula 
\begin{equation}
\label{integ}
\begin{array}{lll}
\displaystyle\int_{0}^{T} \left({^{^C}\mathcal{D}_{0^+}^{^\alpha}}u(t)\right)v(t)dt 
= \displaystyle\int_{0}^{T} u(t)\left({^{^{RL}}\mathcal{D}_{T^-}^{^\alpha}} v(t)\right)dt
+ u(T)\lim\limits_{t\rightarrow T^-}
\mathcal{I}_{_{T^-}}^{^{1-\alpha}}v(t)
- u(0)\mathcal{I}_{_{T^-}}^{^{1-\alpha}}v(0)
\end{array}
\end{equation}
of integration by parts holds.
\end{lemma} 

Our next result (Theorem~\ref{th.sens}),
provides a useful characterization 
of gradient $\omega$-strategic sensors.
To prove it, we make use of \eqref{integ} and also
the following fractional Green's formula.

\begin{lemma}[Fractional Green's formula \cite{zguaid.2021}]	
\label{lem:FGF}
For any $f \in H^{^2}\left(0,T;E\right)$ one has
\begin{equation}
\label{green}
\begin{split}
\displaystyle\int_{0}^{T}
&\int_{\Omega}\left({^{^C}\mathcal{D}_{0^+}^{^\alpha}}r(y,s) 
+ Ar(y,s)\right)f(y,s)dyds\\ 
&=\displaystyle \int_{\Omega}r(y,T)\lim\limits_{s\rightarrow T^-}
\mathcal{I}_{_{T^-}}^{^{1-\alpha}}f(y,s)dy 
- \displaystyle \int_{\Omega}r(y,0)\mathcal{I}_{_{T^-}}^{^{1-\alpha}}f(y,0)dy\\
&\quad + \displaystyle\int_{0}^{T}\int_{\Omega}\left({^{^{RL}}\mathcal{D}_{T^-}^{^\alpha}} 
f(y,s) + A^*f(y,s)\right)r(y,s)dyds \\	
& \quad + \displaystyle\int_{0}^{T}\int_{\partial\Omega}\left(r(\varsigma,s)
\dfrac{\partial f(\varsigma,s)}{\partial\nu_{_{A^*}}} 
- \dfrac{\partial r(\varsigma,s)}{\partial\nu_{_{A}}} 
f(\varsigma,s)\right)d\varsigma ds.
\end{split}
\end{equation}
\end{lemma}

\begin{theorem}
\label{th.sens}
The sequence $(D_i,f_i)_{_{1\leq i\leq p}}$ 
is gradient $\omega$-strategic 
if, and only if,
$$
\displaystyle\sum_{s=1}^{n}M_{j}^sy_j^s 
= 0_{_{\mathbb{R}^{^p}}} \implies y 
= 0_{_{(L^2(\omega))^{^n}}},
$$
where
$$
M_j^s = \left(\varphi_{j,k}^{i,s}\right)_{\substack{1
\leq i \leq p \\ 1 \leq k \leq r_j}},
$$
$$ 
\varphi_{j,k}^{i,s} 
= \left\{
\begin{array}{ll}
\langle \partial_{x_s} \varphi_{j,k}, f_i\rangle_{_{L^2(D_i)}}, 
\quad \text{ for zonal sensors},\\
\partial_{x_s} \varphi_{j,k}(b_i), 
\quad \text{ for pointwise sensors},
\end{array}\right. 
$$
$$
y_j^s = \left(y_{j_{_1}}^s,\ldots,y_{j_{_{r_j}}}^s\right)^T 
\in \mathbb{R}^{^{r_j}},
$$
$$
y_{_{j_{_k}}}^s = \langle \chi_{_\omega}^*y_s,
\varphi_{j,k}\rangle_{_{E}}, \ 1\leq k \leq r_j,
$$
and
$$
y = \left(y_1,\ldots,y_n\right) \in (L^{^2}(\omega))^{^n}.
$$
\end{theorem}

\begin{proof} 
From Proposition~\ref{prp1}, we have that 
$(D_i,f_i)_{_{1\leq i\leq p}}$ is gradient $\omega$-strategic if,
and only if, 
$$
\langle\left(\chi_{_\omega}^n\right)^{^*}y, \nabla\mathcal{K}_\alpha^*z 
\rangle_{_{E^{^n}}} = 0, \ 
\forall z\in L^2(0,T;\mathcal{O}) 
\quad \implies \quad y=0_{_{\left(L^2(\omega)\right)^n}},
$$
which means, by using Lemma~\ref{lemma.new}, that:
\begin{equation}
\label{***}
\displaystyle\sum_{s=1}^{n} \langle \chi_{_\omega}^*y_s,\partial_{x_s}
\mathcal{I}^{1-\alpha}_{_{T^-}}r(0) \rangle_{_{E}} =0 
\quad \implies \quad y=0_{_{\left(L^2(\omega)\right)^n}},
\end{equation}
where $r$ is the solution of \eqref{sys.RL}. Let us now find 
the exact expression of $\langle \chi_{_\omega}^*y_s,
\partial_{x_s}\mathcal{I}^{1-\alpha}_{_{T^-}}r(0) \rangle_{_{E}}$.
Let $s$ be an element in $\segN{1}{n}$. We introduce the system:
\begin{equation}
\label{sys.prf}
\left\{\begin{array}{llll}
^{^C}\mathcal{D}_{0^+}^{^\alpha}\phi(x,\tau) 
=  A\phi(x,\tau),  & (x,\tau)\in Q_T, \ \alpha\in]0,1], \\ 
\phi(\varsigma,\tau) = 0,  & (\varsigma,\tau)\in \Sigma_T, \\
\phi(x,0) = \chi_{_\omega}^*y_s(x),  & x\in\Omega.
\end{array}\right.
\end{equation}
Its unique mild solution is written as:
$$
\phi(\cdot,\tau) = \mathcal{S}_\alpha(\tau)\chi_{_\omega}^*y_s(\cdot) 
= \displaystyle\sum_{j=1}^{\infty}
E_\alpha(-\lambda_j\tau^\alpha)\sum_{k=1}^{r_j}\langle 
\chi_{_\omega}^*y_s,\varphi_{j,k}\rangle_{_{E}}\varphi_{j,k}(\cdot).
$$
Multiplying both sides of \eqref{sys.RL} by 
$\partial_{x_s}\phi$ and integrating 
over $\mathcal{Q}_T =\Omega\times[0,T]$, we get that:
\begin{equation}
\label{prf.1}
\int_{\mathcal{Q}_T}{^{^{RL}}
\mathcal{D}_{T^-}^{^\alpha}}r(x,\tau)\partial_{x_s} \phi(x,\tau)dxd\tau 
= \int_{\mathcal{Q}_T} A^*r(x,\tau) \partial_{x_s} \phi(x,\tau)dxd\tau 
+ \int_{\mathcal{Q}_T}C^*z(\tau)\partial_{x_s} \phi(x,\tau)dxd\tau.
\end{equation}
On the other hand, equation \eqref{integ} gives:
\begin{equation}
\label{prf.2}
\int_{\mathcal{Q}_T}{^{^{RL}}\mathcal{D}_{T^-}^{^\alpha}}
r(x,\tau)\partial_{x_s} \phi(x,\tau)dxd\tau
= \int_{\mathcal{Q}_T}r(x,\tau) A\partial_{x_s} \phi(x,\tau)dxd\tau 
+ \int_{\Omega}\phi(x,0)\partial_{x_s}
\mathcal{I}_{_{T^-}}^{^{1-\alpha}}r(x,0)dx.
\end{equation}
From equations \eqref{green}, \eqref{prf.1}, and \eqref{prf.2}, 
and using the boundary conditions, we obtain that:
\begin{equation}
\label{*}
\begin{split}
\langle \chi_{_\omega}^*y_s,\partial_{x_s}
\mathcal{I}^{1-\alpha}_{_{T^-}}r(0) \rangle_{_{E}}	
&= \int_{\Omega}\phi(x,0)\partial_{x_s}
\mathcal{I}_{_{T^-}}^{^{1-\alpha}}r(x,0)dx,\\
&= \int_{\mathcal{Q}_T} C^*z(\tau)\partial_{x_s} \phi(x,\tau)dxd\tau,\\
&= \int_{0}^{T}\langle z(\tau),C\partial_{x_s} 
\phi(\cdot,\tau)\rangle_{_{\mathbb{R}^{^p}}}d\tau.
\end{split}
\end{equation}
Without loss of generality, we continue the proof for the case of zonal 
sensors (the same can be easily done for pointwise sensors). We have that: 
\begin{equation}
\label{**}
C\partial_{x_s} \phi(\cdot,t) 
= \displaystyle\sum_{j=1}^{+\infty}
\sum_{k=1}^{r_j}E_\alpha(\lambda_jt^\alpha)\langle \chi_{_\omega}^*y_s,
\varphi_{j,k}\rangle_{_{E}}\left(
\begin{array}{c}
\langle \partial_{x_s} \varphi_{j,k},f_1\rangle_{_{L^2(D_1)}}\\
\langle \partial_{x_s} \varphi_{j,k},f_2\rangle_{_{L^2(D_2)}}\\ 
\vdots \\
\langle \partial_{x_s} \varphi_{j,k},f_p\rangle_{_{L^2(D_p)}}
\end{array}\right).
\end{equation}
Using \eqref{***}, \eqref{*} and \eqref{**}, we deduce that 
$(D_i,f_i)_{_{1\leq i\leq p}}$ is gradient 
$\omega$-strategic if, and only if,
\begin{multline*}
\displaystyle\int_{0}^{T}\left\langle z(t) , 
\displaystyle\sum_{s=1}^{n}\sum_{j=1}^{+\infty}\sum_{k=1}^{r_j}
E_\alpha(\lambda_jt^\alpha)  \langle \chi_{_\omega}^*y_s,
\varphi_{j,k}\rangle_{_{E}}\left(
\begin{array}{c}
\langle \partial_{x_s} \varphi_{j,k},f_1\rangle_{_{L^2(D_1)}}\\
\langle \partial_{x_s} \varphi_{j,k},f_2\rangle_{_{L^2(D_2)}}\\ 
\vdots \\
\langle \partial_{x_s} \varphi_{j,k},f_p\rangle_{_{L^2(D_p)}}
\end{array}\right)
\right\rangle_{{\mathbb{R}^p}}dt = 0,\\
\forall z\in L^2(0,T;\mathcal{O}) 
\implies y=0_{_{\left(L^2(\omega)\right)^n}}.
\end{multline*}
From Lemma~5 in \cite{grad.frac}, 
we get that the last expression is equivalent to,
$$
\displaystyle\sum_{s=1}^{n}\sum_{j=1}^{+\infty}
\sum_{k=1}^{r_j}E_\alpha(\lambda_jt^\alpha)  \langle
\chi_{_\omega}^*y_s,\varphi_{j,k}\rangle_{_{E}}\left(
\begin{array}{c}
\langle \partial_{x_s} \varphi_{j,k},f_1\rangle_{_{L^2(D_1)}}\\
\langle \partial_{x_s} \varphi_{j,k},f_2\rangle_{_{L^2(D_2)}}\\ 
\vdots\\
\langle \partial_{x_s} \varphi_{j,k},f_p\rangle_{_{L^2(D_p)}}
\end{array}\right) 
=0, \quad \forall t\in [0,T] 
\implies y=0_{_{\left(L^2(\omega)\right)^n}},
$$
which is also equivalent to,
$$ 
\displaystyle\sum_{j=1}^{+\infty} E_\alpha(\lambda_jt^\alpha)
\sum_{s=1}^{n}M_j^sy_j^s = 0, \quad \forall t\in [0,T] 
\implies y=0_{_{\left(L^2(\omega)\right)^n}}.
$$
Because $E_\alpha(\lambda_jt^\alpha) > 0$ for all $t \in [0,T]$ 
and for all $j \in\mathbb{N^*}$, then we have:
$$ 
\displaystyle\sum_{s=1}^{n}M_j^sy_j^s = 0  
\implies y=0_{_{\left(L^2(\omega)\right)^n}},
$$
and the result is proved.
\end{proof}

The following corollary is an immediate consequence 
of Theorem~\ref{th.sens} in the one-dimensional case, 
i.e., when $n=1$. 

\begin{corollary}
If $n=1$, then $(D_i,f_i)_{_{1\leq i\leq p}}$ is gradient 
$\omega$-strategic if, and only if, 
\begin{itemize}
\item $p \geq \sup\{r_j\}$;
\item $rank\ M_j^1 = r_j$ for all $j\in \mathbb{N}^*$.
\end{itemize} 
\end{corollary}

% ----------------------------------------- 

\section{The Regional Gradient Reconstruction Method}
\label{sec:5}

Now we present the steps of an approach that permits 
the recovery of the initial gradient for \eqref{sys} 
in $\omega$. Our approach is an extension of 
the Hilbert uniqueness method (HUM) for fractional systems.

Let 
$$
K = \left\{ y\in \left(E\right)^n \ | \ y = 0 \ \text{ in } \ 
\Omega\setminus \omega \right\}\cap\left\{\nabla h \ |
\ h\in H^1_0(\Omega)\right\}.
$$

\begin{remark}
\label{rk2}
Note that $(\chi_{_\omega}^{n})^*\chi_{_\omega}^{n}f 
= f , \quad \forall f\in K$.
\end{remark}

For every $\tilde{\varphi}_0$ in $K$, we introduce the system:
\begin{equation}
\label{sys.hum}
\left\{\begin{array}{llll}
^{^C}\mathcal{D}_{0^+}^{^\alpha}\varphi(y,s) 
=  A\varphi(y,s),  & (y,s)\in Q_T, \ \alpha\in]0,1], \\ 
\varphi(\varsigma,s) = 0,  & (\varsigma,s)\in \Sigma_T, \\
\varphi(y,0) = \nabla^*\tilde{\varphi}_0(y), & y\in\Omega,
\end{array}\right.
\end{equation} 
which possesses one and only one mild solution in 
$L^2(0,T;\mathcal{D}(A))\cap C(0,T;E)$, written as follows:
\begin{equation}
\label{sol.sys.hum}
\varphi(t) =  \mathcal{S}_\alpha(t)\nabla^*\tilde{\varphi}_0, \ t\in [0,T].
\end{equation}
We associate with $K\times K$ the form:
\begin{equation}
\label{bil}
\begin{array}{lllll}
\langle \cdot, \cdot\rangle_{_K} & : & K\times K 
& \longrightarrow & \mathbb{C}\\
& & (f,h)& \longmapsto & \displaystyle 
\int_{0}^{T}\langle C\mathcal{S}_\alpha(t)\nabla^*f,
C\mathcal{S}_\alpha(t)\nabla^*h\rangle_{_\mathcal{O}}dt,
\end{array}
\end{equation} 
where $\langle \cdot, \cdot \rangle_{_\mathcal{O}}$ 
is the scalar product in $\mathcal{O}$.

\begin{remark}
\label{rq}
The bilinear form $\langle \cdot,\cdot \rangle_{_K}$ satisfies the conjugate 
symmetry and positive properties, i.e., 
$\langle g,f \rangle =\overline{\langle f,g \rangle}$
and $\langle f,f \rangle_{_K} \geq 0$.
\end{remark}

\begin{lemma}
\label{lemma.prod}
If the system \eqref{sys.hum} is approximately G-observable in $\omega$, 
then the bilinear form \eqref{bil} becomes a scalar product on $K$.
\end{lemma}

\begin{proof}
By Remark~\ref{rq}, we only need to show that 
$\langle \cdot,\cdot\rangle_{_K}$ is definite, 
that is, $\langle f,f\rangle_{_K} = 0 \implies f=0$.\\
Let $f$ be an element of $K$. Hence, 
$$
\langle f,f\rangle_{_K} = \displaystyle \int_{0}^{T}\langle
C\mathcal{S}_\alpha(t)\nabla^*f,C\mathcal{S}_\alpha(t)
\nabla^*f\rangle_{_\mathbb{O}}dt = 0,
$$
which implies that:
$$ 
\langle C\mathcal{S}_\alpha(t)\nabla^*f,
C\mathcal{S}_\alpha(t)\nabla^*f\rangle_{_\mathbb{O}}  = 0.
$$
Using Remark~\ref{rk2}, this means that:
$$
C\mathcal{S}_\alpha(t)\nabla^*f 
= C\mathcal{S}_\alpha(t)\nabla^*(\chi_{_\omega}^{n})^*
\chi_{_\omega}^{n}f= 0,
$$
and, since \eqref{sys.hum} is approximately G-observable 
in $\omega$, we have $\chi_{_\omega}^{n}f = 0$.
We conclude that $f = 0 $ in $ \omega$. It follows that
$f =  0$ from the fact that $f \in K$.
\end{proof}

Let $||\cdot||_{_K}$ be the norm on $K$ associated with 
$\langle \cdot,\cdot \rangle_{_K}$, 
and let us denote again by $K$ its completion by the norm $||\cdot||_{_K}$. 
The space $K$ endowed with $||\cdot||_{_K}$ is now a Hilbert space.

We introduce the following auxiliary system:
\begin{equation}
\label{sys.hum.aux}
\left\{\begin{array}{llll}
^{^{RL}}\mathcal{D}_{T^-}^{^\alpha}\Theta(y,s) 
=  A^*\Theta(y,s) - C^*C\varphi(s),  & (y,s)\in Q_T, \ \alpha\in]0,1], \\ 
\Theta(\varsigma,s) = 0,  & (\varsigma,s)\in \Sigma_T, \\
\lim\limits_{s\rightarrow T^-}\mathcal{I}_{_{T^-}}^{^{1-\alpha}}\Theta(y,s) = 0, 
& y\in\Omega,
\end{array}\right.
\end{equation} 
controlled by the solution of \eqref{sys.hum}. 

\begin{remark}
The condition $z(t) = C\varphi(t)$ implies that system \eqref{sys.hum.aux} 
is the adjoint system of \eqref{sys.hum}.
\end{remark}

We now define an operator that associates to every possible candidate 
of the initial gradient in $\omega$, the projection on $K$ 
$\left(\mbox{via the operator } (\chi_{_\omega}^{n})^*\chi_{_\omega}^{n}\right)$ 
of the term $\nabla\mathcal{I}_{_{T^-}}^{1-\alpha}\Theta(0)$,
\begin{equation}\label{lambda}
\begin{array}{lllll}
\Lambda & : & K & \longrightarrow & K, \\
&  & \tilde{\varphi}_0 & \longmapsto 
& (\chi_{_\omega}^{n})^*\chi_{_\omega}^{n}
\nabla\mathcal{I}_{_{T^-}}^{1-\alpha}\Theta(0).  
\end{array}
\end{equation}
This way, the problem of regional gradient reconstruction 
is reduced to a solvability problem of the equation:
\begin{equation}
\label{rec}
\Lambda \tilde{\varphi}_0 = (\chi_{_\omega}^{n})^*\chi_{_\omega}^{n}
\nabla\mathcal{I}_{_{T^-}}^{1-\alpha}\Theta(0),
\end{equation}
which leads to the next result.

\begin{theorem}
\label{th.main}
If system \eqref{sys} is approximately G-observable in $\omega$, 
then equation \eqref{rec} has a unique solution  
$\tilde{\varphi}_0\in K$, which corresponds 
to the initial gradient $\nabla u_0$ in $\omega$. 
\end{theorem}

\begin{proof}
We shall prove that $\Lambda$ is coercive, that is, 
there exists $\sigma>0$ that verifies $\langle\Lambda v,v\rangle_{_K} 
\geq \sigma\|v\|_{_K}$ for all $v\in K$.
We take $\tilde{\varphi}_0$ to be in $K$,  
\begin{equation*}
\begin{array}{lll}
\langle \Lambda \tilde{\varphi}_0,\tilde{\varphi}_0\rangle_{_K} 
&=& \langle (\chi_{_\omega}^{n})^*\chi_{_\omega}^{n}
\nabla\mathcal{I}_{_{T^-}}^{1-\alpha}
\Theta(0),\tilde{\varphi}_0\rangle_{_K}\\
&=& \langle \mathcal{I}_{_{T^-}}^{1-\alpha}\Theta(0),
\nabla^*\tilde{\varphi}_0\rangle_{_K}.
\end{array}
\end{equation*}
From Proposition~3.3 in \cite{zguaid.2021}, we have that 
$$
\mathcal{I}_{_{T^-}}^{^{1-\alpha}}\Theta(0) 
= \displaystyle\int_{0}^{T}\mathcal{S}_\alpha^*(\tau)
C^*C\mathcal{S}_\alpha\tilde{\varphi}_0 d\tau.
$$
Therefore,
\begin{equation}
\label{prf.th}
\begin{array}{lll}
\langle \Lambda \tilde{\varphi}_0,\tilde{\varphi}_0\rangle_{_K} 
&=& \left\langle \displaystyle\int_{0}^{T}\mathcal{S}_\alpha^*(\tau)
C^*C\mathcal{S}_\alpha(\tau)\nabla^*\tilde{\varphi}_0 d\tau,
\nabla^*\tilde{\varphi}_0\right\rangle_{_K}\\
&=& \displaystyle\int_{0}^{T}\langle C\varphi(\tau),
C\varphi(\tau)\rangle_{_\mathcal{O}} d\tau\\
&=& \|C\varphi(\cdot)\|_{_{L^2(0,T;\mathcal{O})}}^2,
\end{array}
\end{equation}
and equation \eqref{rec} possesses only one solution.
\end{proof}

% -----------------------------------------

\section{Applications}
\label{sec:6}

In this section, we take $\Omega = ]0,1[\times]0,1[$. 
Let $\omega\subset\Omega$ be the desired sub-region. 
We consider the following time-fractional system:
\begin{equation}
\label{sys.app}
\left\{\begin{array}{llll}
^{^C}\mathcal{D}_{0^+}^{^\alpha}y(x_1,x_2,t) 
=  \left(\partial_{x_1}^2 + \partial_{x_2}^2 \right)y(x_1,x_2,t),  
& (x_1,x_2,t)\in Q_T, \ \alpha\in]0,1],\\ 
y(\xi_1,\xi_2,t) = 0,  & (\xi_1,\xi_2,t)\in \Sigma_T, \\
y(x_1,x_2,0) = y_0(x_1,x_2), & (x_1,x_2)\in\Omega.
\end{array}\right.
\end{equation} 
Our goal is to illustrate the steps used to recover 
the initial gradient vector $\nabla y_0 = \left(\partial_{x_1}y_0, 
\partial_{x_2}y_0 \right)$ in the sub-region $\omega$. We present 
the method for the two types of sensors introduced in Section~\ref{sec:4}. 
Recall that the eigenvalues and eigenfunctions of 
$\left(\partial_{x_1}^2 + \partial_{x_2}^2 \right)$ are: 
$$
\lambda_{i,j} = -(i^2 + j^2)\pi^2,
$$
and,
$$
\varphi_{i,j}(x_1,x_2) = 2\sin(i\pi x_1)\sin(j\pi x_2).
$$ 

% -----------------------------------------

\subsection{Zonal Sensors}

Let us take a zonal sensor $(D,f)$
with $D = [c_1,c_2]\times[c_3,c_4] \subset \Omega$ and, 
$$
f(x_1,x_2) = \cos(\sqrt{3}\pi x_1)\sin(\sqrt{2}\pi x_2).
$$

\begin{proposition}
\label{prp.zon}
This sensor $(D,f)$ is gradient $\omega$-strategic if, and only if,  
$$
\gamma_1 \langle \chi_{_\omega}^*u_1,\varphi_{i,j}\rangle_{_{E}} 
+ \gamma_2 \langle \chi_{_\omega}^*u_2,\varphi_{i,j}\rangle_{_{E}}
= 0, \ \forall i,j\in \mathbb{N}^*\times\mathbb{N}^* 
\implies (u_1,u_2) = (0_{_{L^2(\omega)}},0_{_{L^2(\omega)}}),
$$
where,
$$
\gamma_1 = i\displaystyle\iint_{D}\cos(i\pi x_1)
\sin(j\pi x_2)f(x_1,x_2)dx_1dx_2,
$$
and,  
$$
\gamma_2 = j\displaystyle\iint_{D}\sin(i\pi x_1)
\cos(j\pi x_2)f(x_1,x_2)dx_1dx_2.
$$
\end{proposition}

\begin{proof}
In this case, $p=1$, $r_j=1$, and $n=2$. Hence,
$$
M_{ij}^{s} = \left( \langle \partial_{x_s}\varphi_{i,j},
f\rangle_{_{L^2(D)}} \right)_{_{1\times 1}} \ 
\mbox{ and } \ u_{ij}^s = \langle \chi_{_\omega}^*u_s,
\varphi_{i,j} \rangle_{_{E}}, 
\ \forall i,j\in\mathbb{N}^*\times\mathbb{N}^*, 
\ s=\left\{1,2\right\}.
$$
One can see that:
$$
\langle \partial_{x_1} \varphi_{i,j},f\rangle_{_{L^2(D)}} 
= 2i\pi\displaystyle\iint_{_D}
\cos(i\pi x_1)\sin(j\pi x_2)f(x_1,x_2)dx_1 dx_2,
$$ 
and,
$$ 
\langle \partial_{x_2} \varphi_{i,j},f\rangle_{_{L^2(D)}} 
= 2j\pi\displaystyle\iint_{D}\sin(i\pi x_1)
\cos(j\pi x_2)f(x_1,x_2)dx_1dx_2.
$$
Hence, using Theorem~\ref{th.sens}, $(D,f)$ 
is gradient $\omega$-strategic if, and only if,
$$
M_{ij}^1u_{ij}^1 + M_{ij}^2u_{ij}^2 =0, 
\ i,j\in \mathbb{N}^*\times\mathbb{N}^* 
\implies (u_1,u_2)= (0_{_{L^2(\omega)}},0_{_{L^2(\omega)}}),
$$
that is, if, and only if,
$$
\gamma_1 \langle \chi_{_\omega}^*u_1,\varphi_{i,j}
\rangle_{_{E}} + \gamma_2 \langle 
\chi_{_\omega}^*u_2,\varphi_{i,j}\rangle_{_{E}} = 0, 
\ \forall i,j\in \mathbb{N}^*\times\mathbb{N}^* 
\implies (u_1,u_2) = (0_{_{L^2(\omega)}},0_{_{L^2(\omega)}}).
$$
The proof is complete.
\end{proof}

Let us now introduce the set:
$$
\tilde{K} = \left\{f \in (E)^2 \ | \ f= 0 
\ \mbox{ in } \ \Omega\setminus\omega\right\}
\bigcap\left\{\nabla h \ | \ h \in H^1_0(\Omega)\right\}.
$$
Using Lemma~\ref{lemma.prod}, for any $h = (h_1,h_2)$ 
and $g = (g_1,g_2)$ in $\tilde{K}$, the expression:
$$
\langle h,g \rangle_{_{\tilde{K}}} 
= \displaystyle\int_{0}^{T}\langle \mathcal{S}_\alpha(t)
\sum_{s=1}^{2}\partial_{x_s} h_s,f \rangle_{_{L^2(D)}}\langle 
\mathcal{S}_\alpha(t)\sum_{s=1}^{2}
\partial_{x_s}g_s,f \rangle_{_{L^2(D)}}dt,
$$
defines a scalar product whenever the system \eqref{sys.app} 
is approximately $G$-observable in $\omega$ and, 
$$
\|h\|_{_{\tilde{K}}} = \left(\displaystyle
\int_{0}^{T}\langle \mathcal{S}_\alpha(t)
\sum_{s=1}^{2}\partial_{x_s} h_s,
f \rangle_{_{L^2(D)}}^2dt\right)^{\frac{1}{2}},
$$
is the associated norm. Keeping in mind formula \eqref{C*z}, 
we can write the adjoint system as follows: 
\begin{equation*}
\left\{\begin{array}{rlll}
^{^{RL}}\mathcal{D}_{T^-}^{^\alpha}\psi_1(x_1,x_2,t) 
&=& \left(\partial_{x_1}^2 + \partial_{x_2}^2 \right)
\psi_1(x_1,x_2,t)  & (x_1,x_2,t)\in Q_T,\\ 
&  &  - \chi_{_{D}}f(x_1,x_2)\langle \mathcal{S}_\alpha(t)
\sum_{s=1}^{2}\partial_{x_s} 
(\chi_{_\omega}^* h_s),f \rangle_{_{L^2(D)}}, & \alpha\in]0,1],\\
\psi_1(\xi_1,\xi_2,t)& =& 0,  & (\xi_1,\xi_2,t)\in \Sigma_T, \\
\lim\limits_{t\rightarrow T^-}\mathcal{I}_{_{T^-}}^{^{1-\alpha}}
\psi_1(x_1,x_2,t) &=& 0, & (x_1,x_2)\in\Omega.
\end{array}\right.
\end{equation*} 
It follows from Theorem~\ref{th.main} that the equation 
$\Lambda h = (\chi_{_\omega}^{n})^*\chi_{_\omega}^{n}
\nabla\mathcal{I}_{_{T^-}}^{1-\alpha}\psi_1(0)$ 
possesses one and only one solution in $\tilde{K}$.

% -----------------------------------------

\subsection{Pointwise Sensors}

Now we reconsider system \eqref{sys.app} but
augmented with the output: 
\begin{equation}
\label{exp.p}
z(t) = y(b_1,b_2,t),
\end{equation}
where $(b_1,b_2)$ is the sensor location. 
Hence, from \eqref{sol} and \eqref{exp.p}, we have,
$$
C\mathcal{S}_\alpha(t)y_0 
= \displaystyle\sum_{i,j=1}^{+\infty}
E_{\alpha}(-\lambda_{i,j}t^{\alpha})\langle y_0,
\varphi_{i,j}\rangle\varphi_{i,j}(b_1,b_2).
$$
Note that the pointwise sensor has an unbounded observation operator. 
Since $|\varphi_{i,j}|\leq 2$ for all $i,j\in\mathbb{N}^*\times\mathbb{N}^*$, 
$E_\alpha(\cdot)$ is continuous and $\exists C>0$ such that 
$|E_\alpha(-\lambda_{i,j}t^\alpha)|\leq \dfrac{C}{1 + |\lambda_{i,j}|t^\alpha}$ 
(see \cite{regional.analysis}). Therefore, the admissibility 
condition \eqref{adm.cnd} is satisfied for the pointwise sensor. 

\begin{proposition}
The pointwise sensor $\left((b_1,b_2),\delta_{(b_1,b_2)}\right)$ 
is gradient $\omega$-strategic if, and only if, 
\begin{multline*}
i\cos(i\pi b_1)\sin(j\pi b_2) \langle \chi_{_\omega}^*u_1,
\varphi_{i,j}\rangle_{_{E}} + j\sin(i\pi b_1)
\cos(j\pi b_2) \langle \chi_{_\omega}^*u_2,\varphi_{i,j}\rangle_{_{E}} = 0,\\
\forall i,j\in \mathbb{N}^*\times\mathbb{N}^*
\implies (u_1,u_2) = (0_{_{L^2(\omega)}},0_{_{L^2(\omega)}}).
\end{multline*}
\end{proposition}

\begin{proof}
Similar to the proof of Proposition~\ref{prp.zon}.
\end{proof}

For any $h=(h_1,h_2)$ in $\tilde{K}$, if the system \eqref{sys.app} 
is approximately $G$-observable in $\omega$, then: 
$$
\|h\|_{_{\tilde{K}}} 
= \left( \displaystyle\int_{0}^{T} \left(\mathcal{S}_\alpha(t)
\sum_{s=1}^{2} \partial_{x_s}(\chi_{_\omega}^*h_s)\right)^2(b_1,b_2) 
dt\right)^{\frac{1}{2}},
$$
defines a norm in $\tilde{K}$. Let us write the adjoint system:
\begin{equation*}
\left\{\begin{array}{rlll}
^{^{RL}}\mathcal{D}_{T^-}^{^\alpha}\psi_2(x_1,x_2,t) 
&=& \left(\partial_{x_1}^2 + \partial_{x_2}^2 \right)\psi_2(x_1,x_2,t)  
& (x_1,x_2,t)\in Q_T,\\ 
&  &  - \delta_{(b_1,b_2)}(x_1,x_2)\left(\mathcal{S}_\alpha(t)
\displaystyle\sum_{s=1}^{2}\partial_{x_s}(\chi_{_\omega}^*h_s)\right)(b_1,b_2), 
& \alpha\in]0,1],\\
\psi_2(\xi_1,\xi_2,t)& =& 0,  & (\xi_1,\xi_2,t)\in \Sigma_T, \\
\lim\limits_{t\rightarrow T^-}\mathcal{I}_{_{T^-}}^{^{1-\alpha}}\psi_2(x_1,x_2,t) 
&=& 0, & (x_1,x_2)\in\Omega.
\end{array}\right.
\end{equation*} 
It follows from Theorem~\ref{th.main} that the equation 
$\Lambda h = (\chi_{_\omega}^{n})^*\chi_{_\omega}^{n}
\nabla\mathcal{I}_{_{T^-}}^{1-\alpha}\psi_2(0)$ 
has one and only one solution.

% -----------------------------------------

\section{Numerical Simulations}
\label{sec:7}

In this section, we illustrate the adopted method for solving 
the gradient reconstruction problem by presenting
two examples that show its efficiency. In order to solve 
equation \eqref{rec}, we calculate the components of the operator 
$\Lambda$ for some orthonormal basis 
$\left\{\overline{\varphi}_i\right\}_{_{i\in \mathbb{N}^*}}$ 
of $E^n$, denoted by:
$$
\Lambda_{ij}:= \langle \Lambda\overline{\varphi}_i,
\overline{\varphi}_j\rangle_{_{(E)^n}}. 
$$
We know that $\left\{\varphi_{i}\right\}_{_{i\in \mathbb{N}^*}}$ 
is an orthonormal basis of $E$. Then, by setting 
$\overline{\varphi}_{i,k} = \left(0,\ldots,\varphi_{i},0,\ldots,0\right)\in E^{^n}$, 
where $\varphi_{i}$ is at the $k$-th place, we have that 
$\left\{\overline{\varphi}_{i,k}\right\}_{\substack{i\geq 1 \\ 1\leq k \leq n }}$ 
is an orthonormal basis of $E^n$. From now on, by rearranging the terms, we denote
$\left\{\overline{\varphi}_{i,k}\right\}_{\substack{  i\geq 1 \\ 1\leq k \leq n }}$ 
by $\left\{\overline{\varphi}_{i}\right\}_{i\in\mathbb{N}^*}$. This is possible 
since the mapping:
$$
\begin{array}{lllll}
g & : & \mathbb{N}^*\times \segN{1}{N} 
& \longrightarrow & \mathbb{N}^*,\\
& & \hfill (q,d)& \longmapsto & n(q-1)+d,
\end{array}
$$  
is one to one. The equation \eqref{rec} can now be approximated by 
\begin{equation}
\label{app.rec}
\displaystyle\sum_{l=1}^{M}\Lambda_{il}\tilde{\varphi}_{0,l} 
= \tilde{\Theta}_{i}, \quad i = 1,\ldots,M, 
\end{equation}
with $M\in\mathbb{N}^*$, $\tilde{\varphi}_{0,l} 
= \langle \tilde{\varphi}_0, \overline{\varphi}_l \rangle_{_{E^n}}$, 
and $\tilde{\Theta}_i = \langle (\chi_{_\omega}^{n})^*\chi_{_\omega}^{n}
\nabla\mathcal{I}_{_{T^-}}^{1-\alpha}
\Theta(0),\overline{\varphi}_i \rangle_{_{(E)^n}}$.
We know that:
\begin{equation}
\label{CC}
C\mathcal{S}_\alpha(t)\nabla^*\overline{\varphi}_i 
= \displaystyle\sum_{k,l = 1}^{\infty} E_\alpha(-\lambda_{k,l}
t^\alpha) \langle \nabla^*\overline{\varphi}_i,
\varphi_{k,l} \rangle_{_{E}}C\varphi_{k,l},
\end{equation}
and, from  \eqref{prf.th} and \eqref{CC}, we obtain that:
\begin{equation*}
\begin{array}{llll}
\langle \Lambda\overline{\varphi}_i,\overline{\varphi}_j
\rangle_{_{(E)}}& = & \displaystyle\sum_{k,l,r,s=1}^{\infty}
\int_{0}^{T} E_\alpha(-\lambda_{k,l}t^\alpha)
E_\alpha(-\lambda_{r,s}t^\alpha)dt \langle \nabla^*\overline{\varphi}_i,
\varphi_{k,l} \rangle_{_{E}}\langle \nabla^*\overline{\varphi}_j,
\varphi_{r,s} \rangle_{_{E}}C\varphi_{k,l}C\varphi_{r,s}.
\end{array}
\end{equation*}
To sum up, the reconstruction method is given
by the Algorithm~\ref{alg1}.

\medskip

\begin{algorithm}[H]	
\SetAlgoLined
\vspace*{0.5cm}
\begin{description}
\item[\textbf{Step 1. }] Initialization of Data: 
threshold accuracy $\varepsilon$, the initial guess 
of $\tilde{\varphi}_0$, $\alpha$, sensors properties.
		
\item[\textbf{Step 2. }] Repeat:
		
\begin{description}
\item[\textbullet] Solve \eqref{sys.hum.aux} 
and get $\mathcal{I}_{_{T^-}}^{1-\alpha}\Theta(0)$.
			
\item[\textbullet] Get the components of $\Lambda$ 
($\Lambda_{ij}$) and $\tilde{\Theta}_i(0)$.
			
\item[\textbullet] Solve \eqref{app.rec} and obtain 
$\tilde{\varphi}_{0,j}$, then get $\tilde{\varphi}_0$.
\end{description}
Until: $\|z(\cdot) -C\varphi(\cdot)\|_{_{L^{^2}(0,T;\mathcal{O})}} \leq \varepsilon$.
	
\item[\textbf{Step 3.}] Take $\tilde{\varphi}_0$ 
to be the reconstructed initial gradient vector in $\omega$.
\end{description}
\caption{\label{alg1} Solution to the gradient reconstruction problem.}
\end{algorithm}

% -----------------------------------------

\subsection{Example with a Pointwise Sensor}
\label{sec:ex1}

In our first example, we take $\Omega = ]0,1[$ 
and we consider the following time-fractional system:
\begin{equation}
\label{sys.exp.1}
\left\{\begin{array}{llll}
^{^C}\mathcal{D}_{0^+}^{^{0.84}}x(y,s) 
=  \partial_y^{^2}x(y,s),  & (y,s)\in Q_1, \\ 
x(0,s)=x(1,s) = 0,  & s\in[0,1], \\
x(y,0) = x_0(y), & y\in\Omega,
\end{array}\right.
\end{equation}
where the output function corresponds with the sensor 
$(b,\delta_b)$ with $b= \{0.2\}$. We set $\omega = [0.00\ ,\ 0.25 ]$ 
to be the desired subregion and $g(y) = 2\pi\left(\cos(y\pi)^2 
- \sin(y\pi)^2\right)\cos(y\pi)\sin(y\pi)$ to be the initial 
gradient vector that will be reconstructed in $\omega$, whereas 
the initial state, supposedly unknown, is 
$x_0(y) = \left(\cos(y\pi)\sin(y\pi)\right)^2$. After implementing 
the proposed algorithm (Algorithm~\ref{alg1}), we obtain the 
reconstructed initial gradient $\tilde{\varphi}_0$. As we can see 
in Figure~\eqref{point}, the two graphs of the initial gradient 
vector and the recovered one are neighbor to one another 
in the desired region $\omega$ with a reconstruction error:
$$
\|g - \tilde{\varphi}_0 \|^{^2}_{_{L^{^2}(\omega)}} = 9.47\times10^{-4}.
$$
% -----------------------------
\begin{figure}
\center{\includegraphics[width=\textwidth]{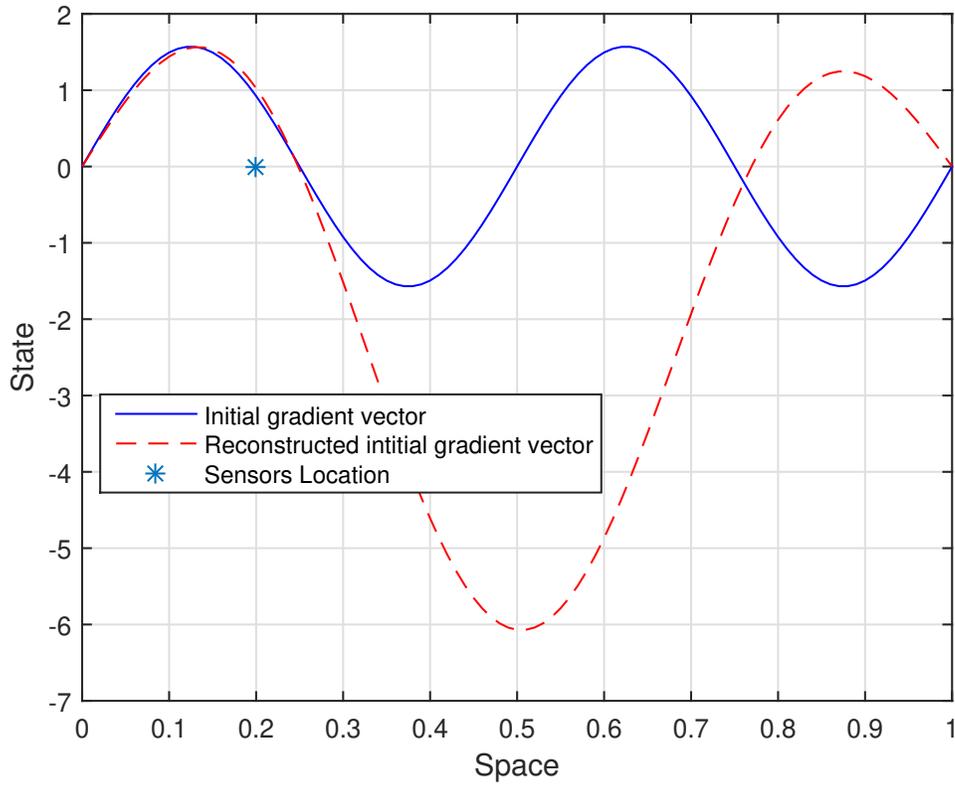}}
\caption{\label{point} The initial gradient vector 
and the reconstructed one in $\Omega$ for the example
of Section~\ref{sec:ex1}.}
\end{figure}
% -----------------------------
This shows that the numerical approach is successful. 
We remark that the proposed algorithm does not put 
into consideration the value of the initial gradient 
outside of $\omega$.
  
Figure~\ref{point.err} portrays the manner in which 
the error evolves in terms of the placement of the sensor. 
As it is seen, there are many positions where the error 
is large or even explodes to infinity. In this case, we say 
that the sensor is non-strategic in $\omega$. Moreover, it is 
clear that the optimal location of the sensor, in the sense 
that it gives the minimum value of the reconstruction error, 
is $b=0.2$. 
% -----------------------------
\begin{figure}
\center{\includegraphics[width=\textwidth]{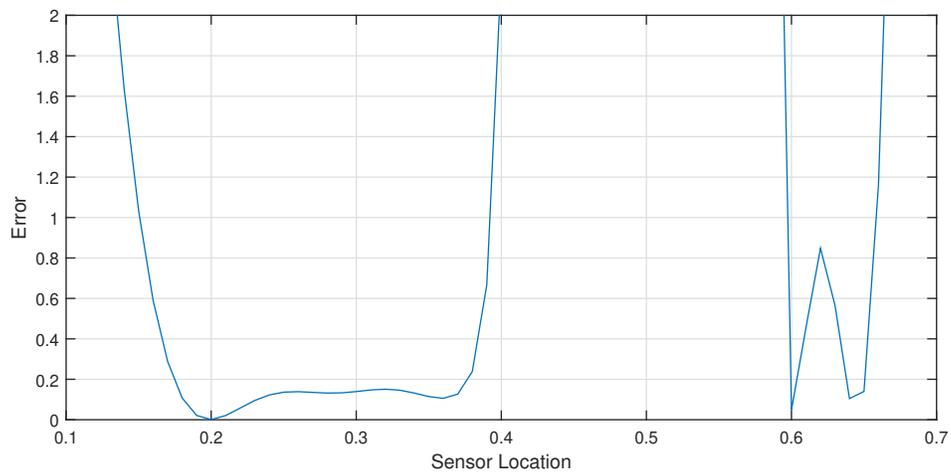}}
\caption{\label{point.err} Reconstruction error \emph{versus} 
sensors location for the example of Section~\ref{sec:ex1}.}
\end{figure}

% -----------------------------------------

\subsection{Example with a Zonal Sensor}
\label{sec:ex2}

Let us now consider the following fractional system:
\begin{equation}
\label{sys.exp.2}
\left\{
\begin{array}{llll}
^{^C}\mathcal{D}_{0^+}^{^{0.5}}x(y,s) 
= \partial_y^2x(y,s),  & (y,s)\in Q_2, \\ 
x(0,s)=x(1,s) = 0,  & s\in[0,2], \\
x(y,0) = x_0(y), & y\in\Omega,
\end{array}\right.
\end{equation}
and take the measurements with a zonal sensor $(D,f)$ 
with $D = [0.9\ ,\ 1.0]$, $f=\chi_{_{D}}$, and the subregion 
$\omega = [0.35\ ,\ 0.65]$. We choose $x_0(y) = (y(1-y))^2$ 
to be the initial state and the gradient to be recovered as
$g(y) = 2y(1-y)(1-2y)$, which are both supposed to be unknown. 
We see in Figure~\ref{zone} that the plot of the initial gradient 
vector is nearly identical to the plot of the reconstructed initial 
gradient. In fact, the reconstruction error takes the value:
$$
\|g - \tilde{\varphi}_0 \|^2 = 1.26\times10^{-6}.
$$
% -----------------------------
\begin{figure}
\center{\includegraphics[width=\textwidth]{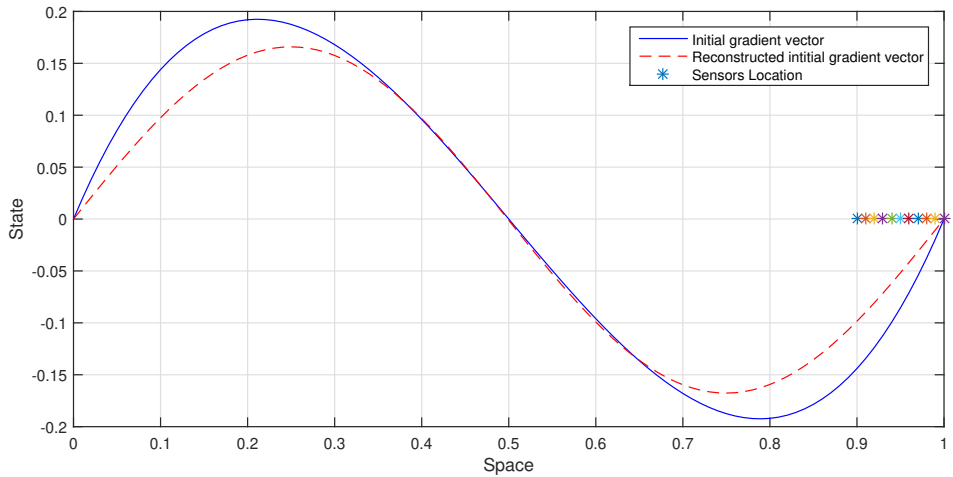}}
\caption{\label{zone} The initial gradient vector 
and the reconstructed one in $\Omega$  
for the example of Section~\ref{sec:ex2}.}
\end{figure}
% -----------------------------
As it can be seen in the two examples, 
the case of a zonal sensor gives numerical results with 
better and smaller reconstruction errors in comparison 
with the case when we consider a pointwise sensor. This might be due to
the fact that, in that case, the sensor has a bounded observation 
operator and a geometrical support with a non-vanishing Lebesgue measure, 
which means that the measurements are given in a much larger set compared 
with the case of a pointwise sensor, where the measurements are provided 
in a single point $b$, meaning that the quantity of 
obtained measurements is much less in this case. Therefore, 
in the case of a bounded observation operator, one has more information 
on the system than the one with an unbounded operator. 
These remarks are based upon the observations 
made during the implementation of the proposed algorithm, 
but more theoretical studies are needed, regarding the theory 
of strategic sensors, to confirm and validate, theoretically,
the observations of our numerical simulations.

% -----------------------------------------

\section{Conclusion}
\label{sec:8}

We dealt with the problem of regional gradient observability 
of linear time-fractional systems given in terms of the Caputo 
derivative. We developed a method that allows us to obtain the 
initial gradient vector in the desired region $\omega$. We also 
gave a complete characterization of the regional gradient observability 
by means of gradient strategic sensors. Even though we studied two particular 
cases of sensors, namely pointwise and zonal, similar results can also 
be obtained for other kinds of sensors, for instance filament ones. 
The numerical simulations presented in this paper are very satisfying 
regarding the error rate and the computation time. We implemented 
the considered examples using the software Matlab R2014b on a 2.5GHz 
core i5 computer with 8 GB of RAM. 

The strength of the HUM approach lies in fact that it can be simulated numerically, 
providing the regional initial gradient with a satisfying control of the error. 
Moreover, it can be adapted to real-world applications. One weakness 
that can be faced while applying this approach to an example 
happens when one considers a dynamic $A$ that possesses an eigenvalue 
with infinite multiplicity. In such a case, one needs an infinite number 
of sensors to observe the system, which can never be achieved in reality.
For future work, we plan to extend the results of this paper 
to the case of semilinear fractional systems. 
Regarding the numerical simulations, we have
considered here some academic examples in order to illustrate 
the obtained theoretical results. We claim that our results 
can be applied and useful to real-world situations, 
a question that will be addressed elsewhere.

% -----------------------------------------

\subsubsection*{Funding}

Partial financial support was received from 
Portuguese Foundation for Science and Technology (FCT) 
within project UIDB/04106/2020 (CIDMA). 

\subsubsection*{Author Contributions}

All authors contributed to this work,
commented on previous versions of the manuscript, 
read and approved the final manuscript.
Conceptualization: Fatima-Zahrae El Alaoui; 
Methodology: Khalid Zguaid, Fatima-Zahrae El Alaoui and Delfim F. M. Torres; 
Formal analysis and investigation: Khalid Zguaid, Fatima-Zahrae El Alaoui and Delfim F. M. Torres; 
Writing - original draft preparation: Khalid Zguaid, Fatima-Zahrae El Alaoui and Delfim F. M. Torres; 
Writing - review and editing: Khalid Zguaid, Fatima-Zahrae El Alaoui and Delfim F. M. Torres; 
Funding acquisition: Khalid Zguaid, Fatima-Zahrae El Alaoui and Delfim F. M. Torres; 
Resources: Khalid Zguaid, Fatima-Zahrae El Alaoui and Delfim F. M. Torres; 
Supervision: Fatima-Zahrae El Alaoui and Delfim F. M. Torres.

\subsubsection*{Acknowledgements}

This research is part of first author's Ph.D., which is carried out 
at Moulay Ismail University, Meknes. Zguaid is grateful to the 
financial support of Moulay Ismail University, Morocco, and CIDMA, Portugal,
for a one-month visit to the Department of Mathematics of University of Aveiro, 
on November and December of 2021. The hospitality of the host institution 
is here gratefully acknowledged. The authors are very grateful to two anonymous 
referees for their suggestions and invaluable comments.

\subsubsection*{Competing Interest}

The authors have no relevant financial or non-financial interests to disclose.

% -----------------------------------------

% -----------------------------------------

\end{document}